\documentclass[11pt]{article}

\usepackage[a4paper,margin=1in]{geometry}

\usepackage[T1]{fontenc}
\usepackage[utf8]{inputenc}
\usepackage{lmodern}

\usepackage{amsmath,amssymb,amsfonts,amsthm}
\usepackage{mathtools}

\usepackage{graphicx}
\usepackage{booktabs}
\usepackage{amssymb}
\usepackage{array}
\usepackage{caption}
\usepackage{subcaption}
\usepackage{xcolor}
\usepackage{algorithm}
\usepackage{algorithmic}
\usepackage{comment}

\usepackage[numbers,sort&compress]{natbib}
\usepackage[colorlinks=true,
            linkcolor=blue,
            citecolor=blue,
            urlcolor=blue]{hyperref}

\newtheorem{theorem}{Theorem}
\newtheorem{lemma}{Lemma}
\newtheorem{proposition}{Proposition}
\newtheorem{corollary}{Corollary}
\newcommand{\R}{\mathbb{R}}
\newcommand{\transpose}{\top}
\theoremstyle{definition}
\newcommand{\Ahat}{\widehat{A}}

\newtheorem{assumption}{Assumption}
\newcommand{\norm}[1]{\left\lVert #1 \right\rVert}

\title{MultiLRSGA: A method for multi-player differentiable games} 

\author{
K. R. Foglia\thanks{Corresponding author. Email: \texttt{katherine.foglia@unical.it}}\\
Department of Mathematics and Computer Science, University of Calabria\\
Ponte P. Bucci, 30B, Arcavacata di Rende (CS), 87036, Italy
\and
V. Colao\\
Department of Mathematics and Computer Science, University of Calabria\\
Ponte P. Bucci, 30B, Arcavacata di Rende (CS), 87036, Italy\\
\texttt{vittorio.colao@unical.it}
\and
A. Borz{\`{\i}}\\
Institut f\"ur Mathematik, Universit\"at W\"urzburg\\
Emil-Fischer-Strasse 30, W\"urzburg, 97074, Germany\\
\texttt{alfio.borzi@mathematik.uni-wuerzburg.de}
}

\date{}

\begin{document}

\maketitle

\begin{abstract}
We propose MultiLRSGA, an $h$-player extension of LRSGA (see \cite{vater2025}) for the computation of stable Nash equilibria in differentiable games. The method originates from the decomposition of the game Jacobian into symmetric and antisymmetric components, which motivates symplectic corrections designed to attenuate the rotational part of the dynamics. In the two-player setting, LRSGA replaces mixed second-order blocks with low-rank secant approximations. The passage to the multi-player case, however, is not a mere blockwise reformulation: the antisymmetric correction is no longer determined by a single pair of cross-interactions, but by a block antisymmetric operator collecting all pairwise couplings among the players. On this basis, we formulate MultiLRSGA by constructing, for each player, a low-rank approximation of the Jacobian of the partial gradient and extracting from it the blocks required to define an approximate antisymmetric correction. Under standard local assumptions around a stable Nash equilibrium, we prove local linear convergence of the method. The key technical ingredient is a lemma controlling the distance between the exact antisymmetric correction and its secant approximation in the $h$-player setting, thereby extending to the multi-player framework the convergence mechanism previously available for LRSGA. The proposed formulation preserves the computational advantages of low-rank symplectic corrections and is naturally suited to numerical validation on differentiable games with explicit payoffs and more than two agents.

\end{abstract}

\medskip


\section{Introduction}

Differentiable games provide a natural framework for learning problems in which several objectives interact and no single scalar function governs the whole dynamics. In this setting, the direct analogy with classical optimization breaks down: even when each player follows its own gradient, the resulting iterations may exhibit persistent oscillations, limit cycles, or a marked sensitivity to the choice of the stepsize. This viewpoint was clarified in particular by Balduzzi et al.~\cite{balduzzi2018} and then by Letcher et al.~\cite{letcher2019}, who showed that the game Jacobian admits a decomposition into a symmetric part, related to potential-like dynamics, and an antisymmetric part, responsible for Hamiltonian or rotational effects.

On top of this decomposition, the recent literature has developed methods that incorporate higher-order information in order to control the rotational component of the game dynamics. Symplectic Gradient Adjustment (SGA) arises precisely from the idea of using the antisymmetric part of the game Jacobian to modify the update direction and favour the attainment of stable fixed points~\cite{balduzzi2018,letcher2019}. Competitive Gradient Descent (CGD), in turn, defines the update as the Nash equilibrium of a regularized local bilinear approximation of the game and yields local convergence guarantees in the competitive two-player setting~\cite{schaefer2019}. In both cases, the message is the same: cross-interactions between players contain information that is dynamically decisive and cannot be ignored if one aims to mitigate the rotational behaviour of gradient-based learning in games.

Our starting point is not SGA itself, but LRSGA, namely a low-rank two-player variant in which the mixed second-order blocks are replaced by rank-one secant approximations~\cite{vater2025}. This construction preserves the symplectic flavour of the correction while avoiding the explicit computation of the mixed derivatives, which can be costly in large-scale models. However, extending LRSGA from two players to the general $h$-player case is by no means automatic. In the two-player setting, the antisymmetric part is governed by a single pair of mixed blocks, so the correction has an essentially binary structure. In the $h$-player setting, by contrast, the antisymmetric correction becomes a block operator in which each block row gathers all interactions between player $i$ and the remaining $h-1$ players. Thus, the correction no longer acts on a single cross-interaction, but rather synthesizes a whole family of pairwise couplings into one coherent skew component.

This structural difference is the main motivation for the present work. The goal is twofold. On the one hand, we introduce a genuine multi-player formulation of LRSGA, called MultiLRSGA, that preserves the computational advantages of low-rank secant corrections. On the other hand, we prove that this formulation is theoretically sound. The key result is a lemma that controls the error between the exact antisymmetric correction and the secant-based correction in the $h$-player case; this lemma makes it possible to transfer the local convergence mechanism already available for the two-player LRSGA method to the multi-player setting. 
The two-player case is therefore recalled only as a minimal technical reference; the real focus of the paper lies in the genuinely multi-player nature of the correction and in the corresponding local theory.


The remainder of the paper is organized as follows.
Section~\ref{sec:setting} introduces the problem setting and the notation used throughout the paper.
In Section~\ref{sec:method}, we formulate MultiLRSGA for differentiable games with \(h\) players.
Section~\ref{sec:convergence} establishes the key approximation lemma for the antisymmetric correction and uses it to extend the local convergence result from the two-player LRSGA setting to the multi-player case.
Section~\ref{sec:experiment} presents a preliminary numerical experiment, providing a first validation of the proposed method.

Finally, Section~\ref{sec:conclusion} concludes the paper by summarizing the main findings and outlining future research directions.

\section{Problem setting and notation}\label{sec:setting}

Let \(S_i\subset \mathbb R^{d_i}\), \(i=1,\ldots,h\), be nonempty, convex, and compact sets, and let
\(
d:=\sum_{i=1}^h d_i.
\)
Let \(\Omega\subset \mathbb R^d\) be an open convex set such that
\(S_1\times\cdots\times S_h\subset \Omega\). An \(h\)-player differentiable game can be formulated as
\begin{equation}\label{eq:game-problem}
\min_{x_i\in S_i} f_i(x_1,\ldots,x_h),\qquad i=1,\ldots,h,
\end{equation}
where \(f_i\in C^3(\Omega,\mathbb R)\) is the objective function of player \(i\).
For brevity, we denote the joint strategy vector by
\(
w=(x_1,\ldots,x_h)\in\mathbb R^d.
\)
A point \((x_1^\star,\ldots,x_h^\star)\in S_1\times\cdots\times S_h\)
is called a Nash equilibrium (NE) of \eqref{eq:game-problem} if no player can
decrease its own objective by a unilateral deviation, that is,
\[
f_i(w^\star)\le f_i(x_i,w^\star_{-i}),
\qquad \forall x_i\in S_i,\qquad i=1,\ldots,h,
\]
where
\(
w^\star_{-i}:=(x_1^\star,\ldots,x_{i-1}^\star,x_{i+1}^\star,\ldots,x_h^\star).
\)
Equivalently,
\(
x_i^\star\in \operatorname*{arg\,min}_{x_i\in S_i}
f_i(x_i,w^\star_{-i}),
\quad i=1,\ldots,h.
\)
Such points will be denoted by \(w^\star\).
We introduce the game gradient

\begin{equation}\label{eq:game-gradient}
F(w) \coloneqq 
\big(
\nabla_{x_1} f_1(w)^\top,
\dots,
\nabla_{x_h} f_h(w)^\top
\big)^\top .
\end{equation}

and the game Hessian
\begin{equation}\label{eq:game-hessian}
H(w) \coloneqq DF(w) = \big(H_{ij}(w)\big)_{i,j=1}^h,
\qquad
H_{ij}(w) \coloneqq \nabla^2_{x_i x_j} f_i(w) \in \R^{d_i \times d_j}.
\end{equation}
As in the mechanics of differentiable games~\cite{balduzzi2018,letcher2019}, we decompose the game Hessian as
\begin{equation}\label{eq:SA-decomposition}
H(w) = S(w) + A(w),
\end{equation}
where
\[
S(w) \coloneqq \frac{1}{2}\big(H(w)+H(w)^\transpose\big),
\qquad
A(w) \coloneqq \frac{1}{2}\big(H(w)-H(w)^\transpose\big).
\]
At block level, this means
\begin{equation}\label{eq:block-A}
A_{ij}(w) = \frac{1}{2}\Big(H_{ij}(w)-H_{ji}(w)^\transpose\Big),
\qquad
A_{ii}(w)=0.
\end{equation}
This notation will make the role of the antisymmetric correction completely explicit.
The first-order stationarity and the second-order necessary conditions for an interior local Nash equilibrium \(w^\star\) are
\begin{equation}\label{equiHess1NE}
F(w^\star)=0, \qquad \nabla_{x_ix_i}^2 f_i(w^\star)\succeq 0, \quad \text{for }i \in\{1, \dots, h\}.
\end{equation}
That is, the game gradient vanishes at $w^\star$ and each player’s second order derivative with respect to its own variable is positive semidefinite.


\section{The MultiLRSGA Method}\label{sec:method}

We start the description of our method by briefly recalling the two-player case. When $h=2$, LRSGA can be viewed as a low-rank counterpart of SGA. Instead of computing the mixed derivatives
\[
\nabla^2_{x_1x_2}f_1(w^k),
\qquad
\nabla^2_{x_2x_1}f_2(w^k),
\]
one builds secant approximations of the corresponding Jacobians and extracts from them the cross-blocks needed in the symplectic correction~\cite{vater2025}. In this way, one keeps the structure of SGA while avoiding the explicit evaluation of the exact mixed derivatives. 
This two-player construction is the point of departure of the present work. Our purpose here is not to revisit the two-player analysis, but to show how the low-rank symplectic idea changes in nature as soon as one moves to the \(h\)-player setting.

\subsection{From a binary correction to a multi-player correction}\label{subsec:binary-multi}

The main structural change is already visible in~\eqref{eq:block-A}. For $h>2$, the antisymmetric part is no longer described by a single pair of off-diagonal blocks. Indeed, for each $i \neq j$,
\[
A_{ij}(w)
=
\frac{1}{2}\Big(\nabla^2_{x_i x_j}f_i(w)-\nabla^2_{x_j x_i}f_j(w)^\transpose\Big),
\]
so the antisymmetric correction depends on all pairwise interactions among the players. 
For each player \(i\), the update aggregates the contributions of the other \(h-1\) players instead of relying on a single mixed block. 
This is the main difference from the two-player case.

\subsection{Construction of MultiLRSGA}\label{subsec:multilrsga-construction}

For each player $i$, we introduce a matrix
\[
M_i^k \in \R^{d_i \times d}
\]
that approximates the Jacobian of the partial gradient $\nabla_{x_i}f_i$ with respect to the whole variable $w$, namely
\[
M_i^k \approx D\big(\nabla_{x_i}f_i\big)(w^k).
\]
We update it by a Broyden-type rank-one secant formula~\cite{broyden1965}:
\begin{equation}\label{eq:broyden-update}
M_i^{k+1}
=
M_i^k
+
\frac{
\Big(\nabla_{x_i}f_i(w^{k+1})-\nabla_{x_i}f_i(w^k)-M_i^k(w^{k+1}-w^k)\Big)(w^{k+1}-w^k)^\transpose
}{
\norm{w^{k+1}-w^k}^2
}.
\end{equation}
Writing $M_i^k$ by block columns as
\[
M_i^k = \big([M_i^k]_1,\dots,[M_i^k]_h\big),
\qquad
[M_i^k]_j \in \R^{d_i \times d_j},
\]
the block $[M_i^k]_j$ approximates $\nabla^2_{x_i x_j}f_i(w^k)$. This allows us to define an approximate antisymmetric correction $\Ahat_k$ without ever forming the exact mixed derivatives:
\begin{equation}\label{eq:approx-A-blocks}
(\Ahat_k)_{ii}=0,
\qquad
(\Ahat_k)_{ij}
=
\frac{1}{2}\Big([M_i^k]_j - [M_j^k]_i^\transpose\Big),
\qquad i \neq j.
\end{equation}
The MultiLRSGA iteration is then given by
\begin{equation}\label{eq:multilrsga-iteration}
w^{k+1}
=
w^k - \eta\big(I-\tau \Ahat_k\big)F(w^k),
\end{equation}
where $\eta>0$ is the main stepsize and $\tau>0$ weights the antisymmetric correction.

In component form, the update of player $i$ reads
\begin{equation}\label{eq:component-update}
x_i^{k+1}
=
x_i^k - \eta\,\nabla_{x_i}f_i(w^k)
+
\frac{\eta\tau}{2}
\sum_{\substack{j=1 \\ j\neq i}}^h
\Big([M_i^k]_j - [M_j^k]_i^\transpose\Big)
\nabla_{x_j}f_j(w^k).
\end{equation}
This formula makes the multi-player nature of the method explicit: for each player, the correction acts as a combination of the gradients of all the other players, mediated by low-rank approximations of the corresponding mixed couplings.

\subsection{Interpretation}\label{subsec:interpretation}

MultiLRSGA can be interpreted as a low-rank multi-player extension of the SGA philosophy. The exact symplectic correction would involve the matrix $A(w^k)$ itself, whereas MultiLRSGA replaces $A(w^k)$ with the secant-based approximation $\Ahat_k$ obtained from the matrices $M_i^k$. 

The gain is computational: one avoids the explicit computation of all mixed second-order derivatives $\nabla^2_{x_i x_j}f_i$. 
The price is theoretical: in the present setting, convergence is not a direct consequence of the two-player analysis through a mere notational change. 
One must quantify how the approximation errors in the matrices $M_i^k$ propagate into the block structure of the antisymmetric correction. 
This is precisely the role of the key lemma in the next section.


\section{Local Convergence Analysis}\label{sec:convergence}

\subsection{Local assumptions and the frozen map}\label{subsec:frozen-map}

Throughout this section, let $w^\star$ denote a stable Nash equilibrium. In particular, we assume
\[
F(w^\star)=0,
\]
that $H(w^\star)$ is invertible, and that its symmetric part $S(w^\star)$ is positive semidefinite. To analyse MultiLRSGA, we introduce the frozen map
\begin{equation}\label{eq:frozen-map}
T^\star_{\eta,\tau}(w)
\coloneqq
w-\eta\big(I-\tau A(w^\star)\big)F(w),
\end{equation}
that is, the map obtained by replacing the varying secant correction $\Ahat_k$ with the exact antisymmetric correction frozen at the equilibrium.

The frozen map will serve as the local reference dynamics. The actual MultiLRSGA iteration~\eqref{eq:multilrsga-iteration} will be compared with~\eqref{eq:frozen-map}; local convergence will follow once the difference between $\Ahat_k$ and $A(w^\star)$ is controlled uniformly in a neighbourhood of $w^\star$.

\begin{assumption}\label{ass:local-lipschitz}
There exists a convex neighbourhood $\mathcal{U}$ of $w^\star$ such that the game gradient $F$ is $L_F$-Lipschitz on $\mathcal{U}$ and, for each $i=1,\dots,h$, the mapping $D(\nabla_{x_i}f_i)$ is $L_i$-Lipschitz on $\mathcal{U}$.
\end{assumption}

\begin{assumption}\label{ass:frozen-contractive}
The frozen map $T^\star_{\eta,\tau}$ is continuously differentiable on $\mathcal U$ and it satisfies
\(
\norm{DT^\star_{\eta,\tau}(w^\star)}_2 < 1.
\)
\end{assumption}

\subsection{Playerwise control of the secant approximations}\label{subsec:playerwise-control}

The first step is to control, for each player, the error between the secant matrix $M_i^k$ and the exact Jacobian $D(\nabla_{x_i}f_i)(w^\star)$.

\begin{proposition}\label{prop:playerwise-secants}
Under Assumption~\ref{ass:local-lipschitz}, for each player $i=1,\dots,h$  and for every $k$ such that $w^k,w^{k+1}\in\mathcal U$, 
the Broyden update~\eqref{eq:broyden-update} satisfies
\begin{equation}\label{eq:playerwise-bound}
\begin{aligned}
\norm{M_i^{k+1}-D(\nabla_{x_i}f_i)(w^\star)}_2
&\leq
\norm{M_i^k-D(\nabla_{x_i}f_i)(w^\star)}_2 \\
&\quad + 2L_i
\max\Big\{\norm{w^{k+1}-w^\star}_2,\norm{w^k-w^\star}_2\Big\}.
\end{aligned}
\end{equation}
\end{proposition}
\begin{proof}
Rewrite the update as
    \begin{align*}
    M_i^{k+1}
= M_i^k\!\left(I_d-\frac{s_ks_k^{\!\top}}{s_k^{\!\top}s_k }\right)
+\frac{\bigl(\nabla_{x_i}f_i(w^{k+1})-\nabla_{x_i}f_i(w^{k}) \bigr)s_k^{\!\top}}{s_k^{\!\top}s_k},
\qquad s_k:=w^{k+1}-w^k.
    \end{align*}
    
By adding and subtract \(D(\nabla_{x_i}f_i(w^\star))\), $D(\nabla_{x_i}f_i(w^{k})) \frac{s_ks_k^{\!\top}}{s_k^{\!\top}s_k }$ and $D(\nabla_{x_i}f_i(w^\star)) \frac{s_ks_k^{\!\top}}{s_k^{\!\top}s_k }$ and by using
\(\bigl\|I_d-\frac{s_ks_k^{\!\top}}{s_k^{\!\top}s_k }\bigr\|_2\le 1\),
the Lipschitz property of \(D(\nabla_{x_i}f_i)\),
and $\norm{w^{k+1} - w^k}_2 \leq \norm{w^{k+1} - w^\star}_2 + \norm{w^k - w^\star}_2 \leq 2 \max\{\norm{w^{k+1}_ - w^\star}_2, \norm{w^k - w^\star}\}_2$ we can follow the proof already done for the particular case $h=2$ (see~\cite{vater2025}, Lemma~5.1) proving \eqref{eq:playerwise-bound}.
\end{proof}
Proposition~\ref{prop:playerwise-secants} is the multi-player analogue of the control already used in the two-player LRSGA analysis: it says that the secant error for each player grows at most linearly with the distance of the current iterates from the equilibrium.

\subsection{The key lemma: from playerwise estimates to the skew correction}\label{subsec:key-lemma}

The real jump from the two-player to the multi-player setting occurs at the level of the antisymmetric correction.

\begin{lemma}\label{lem:skew-error}

Assume that, for some $\delta>0$,
\(
\norm{M_i^k-D(\nabla_{x_i}f_i)(w^\star)}_2 \leq \delta,
\quad i=1,\dots,h.
\)
Then the approximate antisymmetric correction $\Ahat_k$ defined in~\eqref{eq:approx-A-blocks} satisfies the following inequality
\begin{equation}\label{eq:skew-error-bound}
\norm{\Ahat_k-A(w^\star)}_2 \leq (h-1)\delta.
\end{equation}
\end{lemma}

\begin{proof}

Let \(C:=\widehat A_k-A(w^\star)\), which is the block-antisymmetric matrix  
    \[
C \;=\;
\begin{pmatrix}
\mathbf 0_{d_1\times d_1} & C_{12}            & C_{13}            & \cdots & C_{1h}\\
\\
-\,C_{12}^{\!\top}        & \mathbf 0_{d_2\times d_2} & C_{23}            & \cdots & C_{2h}\\
\\
-\,C_{13}^{\!\top}        & -\,C_{23}^{\!\top}        & \mathbf 0_{d_3\times d_3} & \cdots & C_{3h}\\
\vdots                    & \vdots                    & \vdots                    & \ddots & \vdots\\
-\,C_{1h}^{\!\top}        & -\,C_{2h}^{\!\top}        & -\,C_{3h}^{\!\top}        & \cdots & \mathbf 0_{d_h\times d_h}
\end{pmatrix}.
\]
where,
\[
C_{ij}
=\tfrac12\!\left([M_i^k]_j - \nabla_{x_i x_j}^2 f_i(w^\star)\right)
-\tfrac12\!\left([M_j^k]_i - \nabla_{x_j x_i}^2 f_j(w^\star)\right)^{\!\top}
\] for \(i\ne j\).
Since each \([M_i^k]_j\) is a block column of \(M_i^k\), \(\|[M_i^k]_j-\nabla_{x_i x_j}^2 f_i(w^\star)\|_2\le \|M_i^k-D(\nabla_{x_i}f_i(w^\star))\|_2\le \delta\).
Analogously, \(\|([M_j^k]_i - \nabla_{x_j x_i}^2 f_j(w^\star))^{\!\top}\|_2
=\|[M_j^k]_i - \nabla_{x_j x_i}^2 f_j(w^\star)\|_2 \le \delta,\);
hence
\(\|C_{ij}\|_2\le \tfrac12(\delta+\delta)=\delta\).
Now take a block vector \(v=(v_1,\dots,v_h) \in \mathbb{R}^d\) with \(v_j\in\R^{d_j}\) and \(\|v\|_2^2=\sum_j\|v_j\|_2^2\). Then
\begin{equation*}
\resizebox{\textwidth}{!}{$
\|(Cv)_i\|_2
\;\le\;\sum_{j \neq i}\|C_{ij}\|_2\,\|v_j\|_2
\;\le\;\left(\sum_{j \neq i}\|C_{ij}\|_2^{2}\right)^{1/2}\left(\sum_{j \ne i}\|v_j\|_2^{2}\right)^{1/2}
\le\; \left((h-1)\,\delta^2\,\right)^{1/2}\sqrt{\left(\sum_{j \ne i}\|v_j\|_2^{2}\right)},    
$}
\end{equation*}

i.e. \(\|(Cv)_i\|_2^2 \le (h-1)\,\delta^2\,\sum_{j \ne i}\|v_j\|_2^2.\)
Summing over \(i\) we obtain \(\|Cv\|_2^2 = \sum_i\|(Cv)_i\|_2^2 \le (h-1)\delta^2\,\sum_i\sum_{j \neq i}\|v_j\|_2^2 =  (h-1)\delta^2 (h-1)\sum_{j}\|v_j\|_2^{2} = (h-1)^2\delta^2 \|v\|_2^{2}
\), hence \(\|C\|_2\le (h-1)\delta\).
\end{proof}

Lemma~\ref{lem:skew-error} is the key technical ingredient of the extension. In the two-player case, the antisymmetric correction depends on a single pair of mixed blocks, while in the $h$-player case each block row of $\Ahat_k-A(w^\star)$ contains $h-1$ contributions, one for each interaction with the remaining players. Thus the factor $(h-1)$ represents the quantitative trace of the multi-player nature of the correction.
Notice that for $h=2$ Lemma~\ref{lem:skew-error} coincides with the original Lemma~5.2 of~\cite{vater2025}.


\subsection{Local linear convergence}\label{subsec:local-linear-convergence}

We can now state the local convergence result for MultiLRSGA.

\begin{theorem}[Local linear convergence]\label{thm:local-convergence}
Assume that $w^\star$ is a stable Nash equilibrium, that Assumptions~\ref{ass:local-lipschitz} and~\ref{ass:frozen-contractive} hold, and that
\begin{equation}\label{eq:step-condition}
\eta\tau(h-1)L_F < 1.
\end{equation}
Then there exist constants $R>0$ and $\delta_0>0$ such that $B_R(w^\star)\subseteq\mathcal U$ and, if
\[
w^0 \in B_R(w^\star),
\qquad
\norm{M_i^0-D(\nabla_{x_i}f_i)(w^\star)}_2 \leq \delta_0,
\quad i=1,\dots,h,
\]
the sequence generated by MultiLRSGA converges linearly to $w^\star$. More precisely, there exist constants $C>0$ and $q\in(0,1)$ such that
\begin{equation}\label{eq:linear-rate}
\norm{w^k-w^\star}_2 \leq Cq^k,
\qquad k \geq 0.
\end{equation}
\end{theorem}

\begin{proof}
TThe main difference with the particular case \(h=2\) is the factor \((h-1)\) in Lemma~\ref{lem:skew-error}. By assuming $\eta\,\tau\,(h-1)L_F<1$ (instead of $\eta\,\tau\,L_F<1$),
the proof follows as in the two-player case. Equivalently, this hypothesis can be enforced by rescaling one of the step-size parameters as \(\tau \leftarrow \tau/(h-1)\) \quad or \quad \(\eta \leftarrow \eta/(h-1)\).
\end{proof}

The actual iteration~\eqref{eq:multilrsga-iteration} can be written as a perturbation of the frozen dynamics~\eqref{eq:frozen-map}, with perturbation term $\eta\tau\big(\Ahat_k-A(w^\star)\big)F(w^k)$.
By Lemma~\ref{lem:skew-error}, this term is controlled by the secant errors with the factor $(h-1)$, while Assumption~\ref{ass:frozen-contractive} yields the local contraction of the frozen map. 
Condition~\eqref{eq:step-condition}, together with a sufficiently small choice of the initial neighbourhood and of the initial secant errors, ensures that the perturbation does not destroy this contraction. 
This gives the linear estimate~\eqref{eq:linear-rate}.


\begin{corollary}[Local linear convergence under smoothness assumptions]
\label{cor:smooth-local-convergence}
Let $\Omega\subseteq\R^d$ be an open convex domain, let
$f_i\in C^3(\Omega,\R)$ for all $i=1,\dots,h$, and suppose that
$w^\star\in\Omega$ is a stable Nash equilibrium. 
Then there exist $R>0$ and $\delta_0>0$, with $B_R(w^\star)\subseteq\Omega$,
such that, choosing $w^0\in B_R(w^\star)$ and $M_i^0$ satisfying
\[
\norm{M_i^0-D(\nabla_{x_i}f_i)(w^\star)}_2<\delta_0,
\quad i=1,\dots,h,
\]
the MultiLRSGA iterates generated by~\eqref{eq:multilrsga-iteration}
converge linearly to $w^\star$ for appropriately chosen parameters
$\eta,\tau>0$ satisfying the hypotheses of
Theorem~\ref{thm:local-convergence}.
\end{corollary}


The proof follows the same lines as the proof of the corresponding local convergence
corollary in the two-player LRSGA analysis. The contractivity of the frozen map is
obtained from the same SGA monotonicity argument used in Lemma~4.3 of
\cite{vater2025}. Although that result is stated for two-player games, its proof only relies
on the matrix decomposition
\(
H(w^\star)=S(w^\star)+A(w^\star),
\quad
A(w^\star)^\top=-A(w^\star),
\)
together with the positive semidefiniteness of \(S(w^\star)\) and the nonsingularity of
\(H(w^\star)\). Hence the same argument applies to the full \(d\times d\) game Hessian
of the present \(h\)-player game. Therefore, \(\eta,\tau>0\) can be chosen so that
Assumption~2 holds. Taking \(\eta\) smaller if necessary, one can also impose
\(
\eta\tau(h-1)L_F<1.
\)
The conclusion then follows from Theorem~\ref{thm:local-convergence}, with
Lemma~\ref{lem:skew-error} replacing the two-player skew-correction estimate and producing
the factor \((h-1)\).

\section{Numerical Experiments}
\label{sec:experiment}
We tested \emph{MultiLRSGA} with a \emph{random} initialization of the mixed second–derivative blocks. We considered three differentiable objectives:
\[
\begin{aligned}
f_1(x_1,x_2,y,z)&=\tfrac12(x_1^2+x_2^2)+\,x_1\,\tanh(y)+(0.9)\,x_2\,\tanh(z),\\
f_2(x_1,x_2,y,z)&=\tfrac12\,y^2-\,y\,\tanh(x_1)+(0.8)\,y\,\tanh(z),\\
f_3(x_1,x_2,y,z)&=\tfrac12\,z^2-(0.9)\,z\,\tanh(x_2)-(0.8)\,z\,\tanh(y).
\end{aligned}
\]
A Nash equilibrium is $w^\star =(0,0,0,0)$ which satisfies the first-order conditions:
\[
\left\{
\begin{aligned}
\nabla_{x_1} f_1 &= x_1+\tanh(y)=0,\\
\nabla_{x_2} f_1 &= x_2+(0.9)\,\tanh(z)=0,\\
\nabla_{y}   f_2 &= y-\tanh(x_1)+(0.8)\,\tanh(z)=0,\\
\nabla_{z}   f_3 &= z-(0.9)\,\tanh(x_2)-(0.8)\,\tanh(y)=0.
\end{aligned}
\right.
\]

The initial point $w_0$ was fixed at $x_0=(1.0,-0.8)$, $y_0=0.9$, $z_0=-0.7$.


We compared our \textbf{MultiLRSGA} setting $\eta=0.001$ and $\tau=1.0$ with the \textbf{Gradient descent with the game gradient}, i.e.,
\[
w_{k+1}=w_k-\eta\,F(w_k),\]
using the same learning rate $\eta=0.001$. At each iteration we plotted the norm of the game gradient $\|F(w_k)\|$, the 3D trajectory in $(\|x_k\|,y_k,z_k)$, and the component norms $\|x_k\|$, $|y_k|$, $|z_k|$.

\begin{figure}[H]
  \centering
  \captionsetup[subfigure]{justification=centering}
  \begin{subfigure}[t]{0.48\linewidth}
    \centering
    \includegraphics[width=\linewidth]{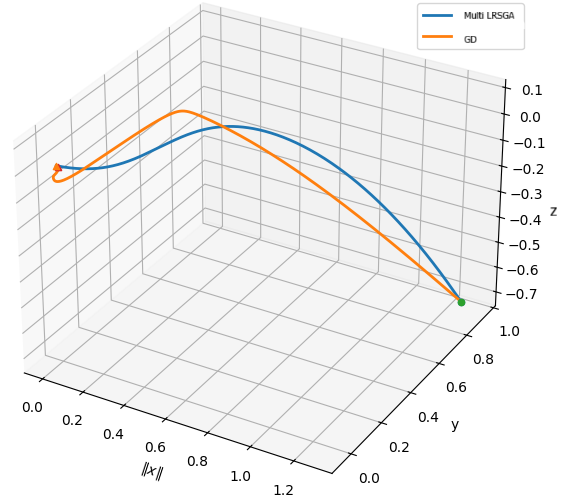}
    \caption{3D trajectory in $(\|x_k\|,y_k,z_k)$: MultiLRSGA vs Gradient descent.}
    \label{fig:traj3d}
  \end{subfigure}\hfill
  \begin{subfigure}[t]{0.48\linewidth}
    \centering
    \includegraphics[width=\linewidth]{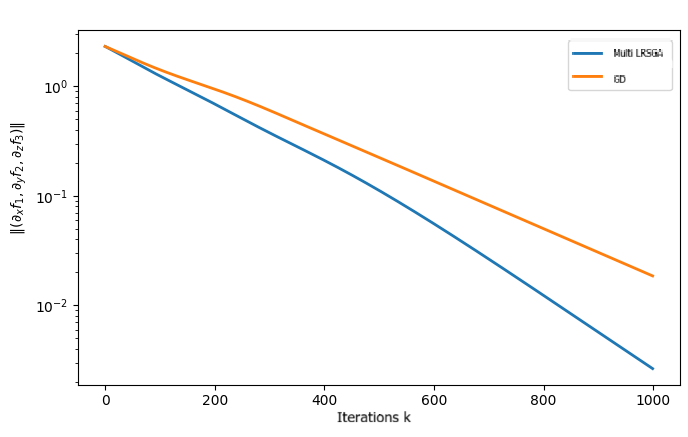}
    \caption{Residual $\|F(w_k)\|$ (log–scale): MultiLRSGA vs Gradient descent ($w_{k+1}=w_k-\eta F(w_k)$).}
    \label{fig:residual}
  \end{subfigure}\\[0.6em]
  \begin{subfigure}[t]{0.48\linewidth}
    \centering
    \includegraphics[width=\linewidth]{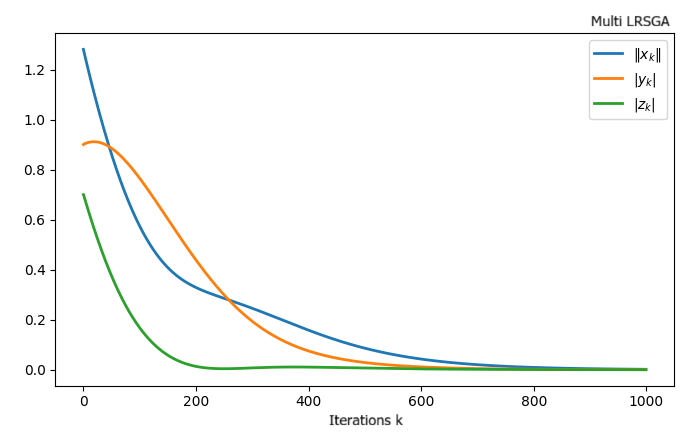}
    \caption{Component norms with \textbf{MultiLRSGA}: $\|x_k\|$, $|y_k|$, $|z_k|$.}
    \label{fig:norms_tau1}
  \end{subfigure}\hfill
  \begin{subfigure}[t]{0.48\linewidth}
    \centering
    \includegraphics[width=\linewidth]{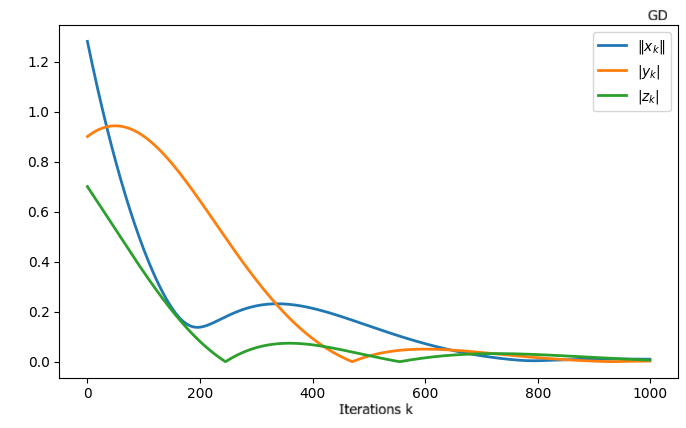}
    \caption{Component norms with \textbf{Gradient descent}: $\|x_k\|$, $|y_k|$, $|z_k|$.}
    \label{fig:norms_tau0}
  \end{subfigure}
  \caption{MultiLRSGA vs Gradient Descent: trajectories, residuals, and component norms.}
  \label{fig:multilrsga_fourplots}
\end{figure}

As shown in Figure \ref{fig:multilrsga_fourplots}, both methods converge to $w^\star$; MultiLRSGA shows a more regular and faster residual decay and visibly reduces oscillations compared with gradient descent.

\section{Conclusion}
\label{sec:conclusion}

We introduced MultiLRSGA, a multi-player extension of LRSGA for differentiable games with $h \geq 2$ players. 
The method constructs playerwise secant approximations of the Jacobians $D(\nabla_{x_i}f_i)$ and uses their block structure to define an approximate antisymmetric correction without explicitly computing the mixed second-order derivatives $\nabla^2_{x_i x_j}f_i$. 
Under local assumptions around a stable Nash equilibrium, we extended the proof strategy of LRSGA to the multi-player setting.

The key step in the extension is the control of the error $\|\Ahat_k-A(w^\star)\|_2$, which quantifies how the playerwise secant errors propagate into the full multi-player skew correction.

The preliminary numerical experiment presented above provides a first validation of the proposed method and illustrates its ability to reduce oscillations with respect to the standard game-gradient dynamics. Future work will focus on a broader numerical validation, including games with more players and more challenging payoff structures.

A further important direction is scalability: although MultiLRSGA avoids the explicit computation of mixed second-order derivatives, the matrices $M_i^k\in\R^{d_i\times d}$ have the same ambient dimensions as the Jacobians they approximate. 
This may become prohibitive in neural-network games or in competitive optimization settings, where the strategies correspond to weights and biases and the number of parameters is very large. 
For this reason, future work will focus on developing a limited-memory variant of MultiLRSGA, in line with the LM-LRSGA method presented in~\cite{lmlrsga}.

\bibliographystyle{plainnat}
\bibliography{references}

\end{document}